\documentclass[10pt]{article}
\usepackage{amsmath,amsthm,amsfonts,amssymb,mathrsfs}
\usepackage{tikz}
\usepackage{graphicx}
\usepackage{hyperref}
\hypersetup{hidelinks}
\usepackage{color}
\usepackage[T1]{fontenc}
\usepackage{amsthm}
\usepackage[margin=1in]{geometry}
\allowdisplaybreaks

\newtheorem{theorem}{Theorem}[section]
\newtheorem{lemma}[theorem]{Lemma}
\newtheorem{corollary}[theorem]{Corollary}
\newtheorem{proposition}[theorem]{Proposition}
\newtheorem{definition}[theorem]{Definition}
\newtheorem{remark}[theorem]{Remark}
\newtheorem{hypothesis}[theorem]{Hypothesis}

\numberwithin{equation}{section}

\title{Continuity of higher-order derivatives for integrated density of states of the discrete Anderson model with respect to the disorder parameter}
\author{
Dhriti Ranjan Dolai$^{1}$
\quad and \quad
Naveen Kumar$^{1}$\\[0.5em]
$^{1}$ Indian Institute of Technology Dharwad,\\
Dharwad 580011, Karnataka, India.\\
\texttt{dhriti@iitdh.ac.in}, \texttt{222071001@iitdh.ac.in}
}

\date{}
\begin{document}
\maketitle

\begin{abstract}
\noindent We derive quantitative continuity estimates for the higher-order derivatives of the integrated density of states (IDS) with respect to the disorder parameter for the Anderson model on $\ell^2(\mathbb{G})$. Here $\mathbb{G}=\mathbb{Z}^d$ or $\mathbb{B}$, where $\mathbb{B}$ denotes the Bethe lattice. Our results hold in the regime of strong disorder, where  entire spectrum is localized. We assume sufficient smoothness of the density of the single site distribution so that the IDS admits higher-order derivatives. More precisely, we establish bounds on the difference between higher-order derivatives of the IDS in terms of the differences in the disorder parameters.
\end{abstract}
\section{Introduction}
The Bethe lattice $\mathbb{B}$ is an infinite connected graph with no closed loops and fixed degree $K+1$ (the number of nearest neighbors) at each vertex. The degree is called the \emph{coordination number}, and the connectivity $K$ equals the coordination number minus one. The distance between two vertices $n_1$ and $n_2$ is denoted by $d(n_1,n_2)$ and is defined as the length of the shortest path connecting them.
We denote by $\ell^2(\mathbb{B})$ the Hilbert space
\[
\ell^2(\mathbb{B})
=
\left\{
u : \mathbb{B}\to\mathbb{C} :
\sum_{n\in\mathbb{B}} |u(n)|^2 < \infty
\right\}.
\]
We fix an arbitrary vertex of $\mathbb{B}$ as the origin and denote it by $0$.
For $n_1,n_2 \in \mathbb{Z}^d$, the distance is defined by $d(n_1,n_2)=\|n_1-n_2\|_1$.
The Hilbert space $\ell^2(\mathbb{Z}^d)$ is defined by
\[
\ell^2(\mathbb{Z}^d)
=
\left\{
u:\mathbb{Z}^d\to\mathbb{C} :
\sum_{n\in\mathbb{Z}^d}|u(n)|^2<\infty
\right\}.
\]
\noindent We consider two Anderson models $H_j^\omega$, $j=1,2$, acting on $\ell^2(\mathbb{G})$, where $\mathbb{G}=\mathbb{B}$ or $\mathbb{Z}^d$, associated with disorder parameters $\lambda_j$. These operators are defined by
\begin{equation}
H_j^\omega = \Delta + \lambda_j V^\omega,
\quad \omega\in\Omega, \quad j=1,2.
\label{anderson}
\end{equation}
Here $\Delta$ denotes the discrete Laplacian acting on $\ell^2(\mathbb{G})$, defined by
\begin{equation*}
(\Delta \psi)(n)
=
\sum_{d(k,n)=1} \psi(k),
\quad
\psi=\{\psi(n)\}_{n\in\mathbb{G}}\in\ell^2(\mathbb{G}).
\end{equation*}
The random potential $V^\omega$ is defined by
\[
(V^\omega \psi)(n)
=
\omega_n\,\psi(n),
\quad n\in\mathbb{G},
\]
where $\{\omega_n\}_{n\in\mathbb{G}}$ is a family of independent and identically distributed non-degenerate real-valued random variables with common distribution $\mu$, referred to as the \emph{single site distribution (SSD)}.

\noindent We consider $H_j^\omega$ as operator-valued random variables on the probability space $(\Omega,\mathcal{B}_\Omega,\mathbb{P})$.
The underlying infinite product probability space
\[
(\Omega,\mathcal{B}_\Omega,\mathbb{P})
=
\left(\mathbb{R}^{\mathbb{G}},\mathcal{B}_{\mathbb{R}^{\mathbb{G}}},
\bigotimes_{n\in\mathbb{G}}\mu\right)
\]
is constructed via the Kolmogorov extension theorem, and we write
$\omega=(\omega_n)_{n\in\mathbb{G}}\in\Omega$.
For each $\omega\in\Omega$, the operator $H_j^\omega$ is self-adjoint on
$\ell^2(\mathbb{G})$, with a common dense core given by finitely supported vectors.
Moreover, the family $\{H_j^\omega\}_{\omega\in\Omega}$ is measurable in the sense that for any
$y,z\in\ell^2(\mathbb{G})$, the map $\omega\mapsto\langle y,H_j^\omega z\rangle
$ from $\Omega$ to $\mathbb{C}$ is measurable for $j=1,2$.
We refer to \cite{CL} for further discussion of measurability properties of
random Schr\"odinger operators.

\noindent Let $\{\delta_n\}_{n\in\mathbb{G}}$ denote the canonical orthonormal basis of
$\ell^2(\mathbb{G})$.
The \emph{density of states measure (DOSm)} $\nu_j$ associated with $H_j^\omega$ is given by
\[
\nu_j(A)
=
\mathbb{E}\!\left[
\langle \delta_0,
E_{H_j^\omega}(A)\delta_0
\rangle
\right],
\quad
A\subset\mathbb{R}\ \text{Borel}.
\]
The corresponding distribution function
$\mathcal{N}_j(x)=\nu_j((-\infty,x])$, $x\in\mathbb{R}$,
is known as the \emph{integrated density of states (IDS)}.
If the measure $\nu_j$ is absolutely continuous with respect to Lebesgue measure,
its Radon--Nikodym derivative is called the \emph{density of states function (DOSf)},
and is denoted by $\frac{d\nu_j}{dx}=g_j(x)=\mathcal{N}'_j(x).$
Assuming that the single site distribution (SSD) $\mu_j$ is absolutely continuous,
the Wegner estimate implies that the density of states measure $\nu_j$ is also absolutely continuous.
We refer to \cite{Kir,AMSW,Ackl} for further details on the integrated density of states (IDS) and the Anderson model.

\noindent We now state our assumptions on the disorder parameters $\lambda_j$ and the single site distribution (SSD) $\mu$.
\begin{hypothesis}
\label{hyp}
\begin{enumerate}
\item[(i)]
We assume that $\lambda_1, \lambda_2 \ge \lambda_0$, where $\lambda_0 \gg 1$ is chosen sufficiently large so that the entire spectrum is localized when the disorder strength is at least $\lambda_0$.
\item[(ii)]
The single site distribution $\mu$ is compactly supported on $[a,b]$, and its density $\frac{d\mu}{dx}=\rho$ belongs to $C_c^{m+1}((a,b))$
for some $m\in\mathbb{N}\cup\{0\}$. Moreover, it satisfies
\[
\sup_{0\le k\le m+1}\|\rho^{(k)}\|_\infty<\infty, \quad \text{where}\quad \rho^{(k)}=\frac{d^k\rho}{dx^k}
\]
\end{enumerate}
\end{hypothesis}

\begin{remark}
\label{est-frc}
Consider the operator $H_\lambda^\omega=\Delta+\lambda V^\omega$. 
Then the spectrum of $H_\lambda^\omega$ is given by 
$\sigma(H_\lambda^\omega)=[-2d,2d]+\lambda[a,b]$ on $\ell^2(\mathbb{Z}^d)$ for a.e.\ $\omega$, 
and $\sigma(H_\lambda^\omega)=[-2\sqrt{K},2\sqrt{K}]+\lambda[a,b]$ on $\ell^2(\mathbb{B})$ for a.e.\ $\omega$. See \cite{Kir,AMSW,Ackl} for more details.
\end{remark}

\noindent Moreover, under condition (i) of Hypothesis~\ref{hyp}, the entire spectrum exhibits
exponential localization. In particular, for some $0<s<1$, the fractional moment criterion holds:
\begin{equation}\label{frc-mnts}
\sup_{\Re(z)\in\mathbb{R},\ \Im(z)>0}
\mathbb{E}\!\left[
\big|\langle\delta_n,(H_\lambda^\omega-z)^{-1}\delta_k\rangle\big|^s
\right]
\le
C(s,\lambda)\,e^{-\xi(s,\lambda)\,d(n,k)},
\quad \forall\, n,k\in\mathbb{Z}^d.
\end{equation}
Fractional-moment localization estimates of the form~\eqref{frc-mnts} were established in the works of Aizenman--Molchanov~\cite{AM}, Aizenman~\cite{aiz}; see also~\cite{st}. In particular, for the Anderson model on $\ell^2(\mathbb{G})$ in the large-disorder regime $\lambda\ge\lambda_0$, these results yield explicit decay rates. More precisely, one may take
\[
\begin{aligned}
\xi(s,\lambda)
&:=
-s\ln\!\left(\frac{2d\,C_{s,\rho}}{\lambda}\right)
\;\text{or}\;
-s\ln\!\left(\frac{(K+1)\,C_{s,\rho}}{\lambda}\right)
\quad\text{when }\mathbb{G}=\mathbb{Z}^d \text{ or } \mathbb{B}, \\
C(s,\lambda)
&:=
\frac{(2\sqrt{2})^s}{(1-s)|\lambda|^s}
\quad\text{for }\mathbb{G}=\mathbb{Z}^d \text{ or } \mathbb{B}.
\end{aligned}
\]
where $0<C_{s,\rho}<\infty$ is a constant depending only on the fractional-moment exponent $s$ and the single site density $\rho$. It follows that
$\xi(s,\lambda)\ge\xi(s,\lambda_0)$ and $C(s,\lambda)\le C(s,\lambda_0)$ for all $\lambda\ge\lambda_0$.
Hence, the fractional moment criterion~\eqref{frc-mnts} can be rewritten with uniform constants as
\begin{equation}\label{frc-mnts1}
\sup_{\Re(z)\in\mathbb{R},\ \Im(z)>0}
\mathbb{E}\!\left[
\big|\langle\delta_n,(H_\lambda^\omega-z)^{-1}\delta_k\rangle\big|^s
\right]
\le
C(s,\lambda_0)\,e^{-\xi(s,\lambda_0)\,d(n,k)},
\quad  \forall\,\lambda\ge\lambda_0.
\end{equation}
Under Hypothesis~\ref{hyp}, the integrated density of states (IDS) $\mathcal{N}_j(\cdot)$ associated with $H_j^\omega$ ($j=1,2$) admits higher-order derivatives.

\begin{remark}
\label{dr-cnt}
Under Hypothesis~\ref{hyp}, together with the fractional-moment criterion~\eqref{frc-mnts}, the results of~\cite{DKM,dh} imply that $\mathcal{N}_j \in C^m(\mathbb{R})$ for $j=1,2$.
\end{remark}
\noindent We denote $\mathcal{N}_j' = g_j$. For $j=1,2$, we define $\mathcal{N}_j^{(k+1)} = g_j^{(k)}, \quad 0 \le k \le m-1$,
\begin{remark}
\label{cpt}
Since the single site distribution (SSD) is a compactly supported absolutely continuous probability measure by [(ii), Hypothesis~\ref{hyp}], the spectrum of $H_j^\omega$ is a deterministic compact set for a.e. $\omega$. Moreover, the density of states measure (DOSm) $\nu_j$ is a compactly supported, absolutely continuous probability measure on $\mathbb{R}$, and we write $d\nu_j = g_j(x)\,dx$. Consequently, $g_j^{(k)} \in C_c(\mathbb{R})$; that is, $g_j^{(k)}$ is a continuous function on $\mathbb{R}$ with compact support for $0 \le k \le m-1$.
\end{remark}

\noindent We are now ready to state the main result.
\begin{theorem}
\label{main}
Let $\mathcal{N}_j(\cdot)$ denote the integrated density of states (IDS) associated with the operator $H_j^\omega$ defined in~\eqref{anderson}, for $j=1,2$.
Assume Hypothesis~\ref{hyp}. Then for $\lambda_1,\lambda_2\in[\lambda_0,\tilde{\lambda}_0]$ we have
\begin{equation}
\label{est-ids}
\sup_{x \in \mathbb{R}}
\big| \mathcal{N}_1^{(k+1)}(x) - \mathcal{N}_2^{(k+1)}(x) \big|
\le
D_k\,|\lambda_1-\lambda_2|^{\frac{m-k-2}{m}},
\quad
0 \le k < m-2,
\end{equation}
where $\mathcal{N}_j^{(k+1)}=\frac{d^{k+1}\mathcal{N}_j}{dx^{k+1}} .$
Here $D_k$ is a positive constant independent of $\lambda_1$ and $\lambda_2$ and depending only on $k$, $\lambda_0$, and $\tilde{\lambda}_0$.
\end{theorem}
\noindent First, Hislop and Marx~\cite{HM} studied the dependence of the integrated density of states on the single site distribution (SSD) for a class of discrete random Schr\"odinger operators and established a quantitative form of continuity in the weak-* topology. Subsequently, Shamis~\cite{MS} developed an alternative approach based on Ky Fan inequalities and obtained a sharp version of the estimate in~\cite{HM}. Applying the results of~\cite{HM,MS} to the model~\eqref{anderson}, together with Hypothesis~\ref{hyp}, \cite[Remark~3.1]{MS}, and Proposition~\ref{kr}, yields
\begin{equation}
\label{olydrv}
\sup_{x\in\mathbb{R}}
\big| \mathcal{N}_1(x)-\mathcal{N}_2(x) \big|
\le
C\, d_{KR}(\mu_1,\mu_2)^{\frac12}
\le
C\,|\lambda_1-\lambda_2|^{\frac12}.
\end{equation}
In this work, our goal is to obtain estimates of the form~\eqref{olydrv} for higher derivatives (when they exist) of $\mathcal{N}_1$ and $\mathcal{N}_2$, as in~\eqref{est-ids}. The estimate~\eqref{est-ids} is not a consequence of~\eqref{olydrv}. In other words, \eqref{est-ids} cannot be obtained by a straightforward application of the results in~\cite{HM,MS}.

\noindent Since higher-order derivatives of the integrated density of states (IDS) $\mathcal{N}_j(\cdot)$ are mostly known in the presence of localization, we refer to~\cite{DKM,dh} and the references therein for further details. Accordingly, we consider the model $H_j^\omega$ defined in~\eqref{anderson} in the high-disorder regime.

\noindent In our proofs, we did not use any finite-volume restriction of $H^\omega_\lambda$; therefore, one may expect that the arguments apply to both $\mathbb{Z}^d$ and the Bethe lattice $\mathbb{B}$.

\section{Proofs}
\noindent In this section, we present proofs of our results. To establish these results, we use Fourier transform techniques. 
We begin by obtaining bounds on higher-order derivatives of the density of a measure using bounds on derivatives of its Stieltjes (Borel) transform.

\noindent Our goal is to obtain an upper bound for $\|g_j^{(k)}\|_\infty$.
Since $g_j^{(k)}$ is continuous with compact support, it follows that
$\|g_j^{(k)}\|_\infty < \infty$.
However, this bound may depend on the disorder parameter $\lambda_j$.
In order to obtain bounds that are independent of $\lambda_j$,
we use a lemma relating derivatives of a function to derivatives of its
Borel transform. Once such bounds are established, they imply sufficient decay of the Fourier transform of the density of states function $g_j$, independently of $\lambda_j$.
This decay plays a key role in the proof of Theorem~\ref{main}.

\begin{lemma}
\label{drv-br}
Let \(g \in L^{1}(\mathbb{R}, dx) \) and let its Borel transform be defined by $F(z) = \int_{\mathbb{R}} \frac{g(x)}{x - z}\, dx .
$ Let \( J \subset \mathbb{R} \) be a bounded interval and assume that
\begin{equation}
\label{fnbr}
\sup_{\Re(z) \in J,\; \Im(z) > 0}
\left|
\Im\!\left(\frac{d^{k}}{dz^{k}} F(z)\right)
\right|
< \infty,
\quad 0 \le k \le m+1.
\end{equation}
Then \( g \in C^{m}(J) \), and for all \( 0 \le k \le m \) we have
\begin{equation}
\label{lemma-est}
\sup_{E \in J} | g^{(k)}(E) |
\le
\sup_{\Re(z) \in J,\; \Im(z) > 0}
\left|
\Im\!\left(\frac{d^{k}}{dz^{k}} F(z)\right)
\right|
+
|J|
\sup_{\Re(z) \in J,\; \Im(z) > 0}
\left|
\Im\!\left(\frac{d^{k+1}}{dz^{k+1}} F(z)\right)
\right|.
\end{equation}
Here \( |J| \) denotes the Lebesgue measure of \( J \).
\end{lemma}

\begin{proof}
The fact that $g \in C^m(J)$ follows from \cite[Lemma A.1]{DKM}.

\noindent To prove the estimate, we first observe, using integration by parts, that
\begin{equation*}
\frac{d^k}{dz^k} F(z)
= \int_{\mathbb{R}} \frac{g^{(k)}(x)}{x - z} \, dx,
\quad 0 \le k \le m.
\end{equation*}

\noindent Since $g^{(k)}$ is continuous on $J$ for $0 \le k \le m$, we obtain
\begin{equation}
\label{cp-im}
\lim_{\epsilon \to 0^+}
\Im\!\left( \frac{d^k}{dz^k} F(E + i\epsilon) \right)
= g^{(k)}(E),
\quad \forall\, E \in J.
\end{equation}

\noindent Since $F(z)$ is analytic on the upper half-plane $\mathbb{C}^+ = \{ z \in \mathbb{C} : \Im(z) > 0 \},
$ the fundamental theorem of calculus implies that for $0 \le k \le m$ and $[E_0,E] \subseteq J$,
\begin{align}
\Im\!\left( \frac{d^k}{dz^k} F(E + i\epsilon) \right)
-
\Im\!\left( \frac{d^k}{dz^k} F(E_0 + i\epsilon) \right)
&=
\frac{d^k}{dx^k} \Im\!\big( F(E + i\epsilon) \big)
-
\frac{d^k}{dx^k} \Im\!\big( F(E_0 + i\epsilon) \big) \nonumber \\
&=
\int_{E_0}^{E}
\frac{d^{k+1}}{dx^{k+1}}
\Im\!\big( F(x + i\epsilon) \big)\, dx \nonumber \\
&=
\int_{E_0}^{E}
\Im\!\left(
\frac{d^{k+1}}{dz^{k+1}} F(x + i\epsilon)
\right) dx .
\end{align}

\noindent Finally, using \eqref{fnbr} and \eqref{cp-im} in the previous identity yields \eqref{lemma-est}.
\end{proof}
\noindent We now derive uniform upper bounds (independent of $\lambda_j$) for the supremum norms of $g_j$ and its derivatives.
\begin{proposition}
\label{up-bd}
Let $H_j^\omega$ be defined in \eqref{anderson} under Hypothesis~\ref{hyp}, and let $\mathcal{N}_j(\cdot)$ denote the integrated density of states (IDS) associated with $H_j^\omega$.
Then, for all $0 \le k \le m-1$ and all $\lambda_j \in [\lambda_0,\tilde{\lambda}_0]$, we have
\begin{equation}
\label{g-est}
\|g_j^{(k)}\|_\infty
\le
C_k\bigl(1+4d+\tilde{\lambda}_0 b-\lambda_0 a\bigr),
\quad
\mathcal{N}_j^{(k+1)} = g_j^{(k)} .
\end{equation}
\end{proposition}

\begin{proof}
We recall that $\mathcal{N}_j(x)
:=
\nu_j((-\infty,x])
=
\mathbb{E}\!\left[
\langle \delta_0, E_{H_j^\omega}((-\infty,x])\,\delta_0 \rangle
\right],\quad x\in\mathbb{R},$
and note that $g_j^{(k)}(x)
=
\mathcal{N}_j^{(k+1)}(x)
\in C_c(\mathbb{R}),
\quad 0\le k\le m-1 .$ Using the spectral theorem, we write the Borel transform of $g_j$ as
\begin{equation}
\label{brt}
F_j(z)
:=
\int_{\mathbb{R}} \frac{g_j(x)}{x-z}\,dx
=
\mathbb{E}\!\left[
\langle \delta_0,(H_j^\omega-z)^{-1}\delta_0\rangle
\right],
\quad \Im(z)>0 .
\end{equation}

\noindent We now use the exponential localization estimate \eqref{frc-mnts1}. Following the proof of \cite[Theorem~3.4]{DKM}, we obtain
\begin{equation}
\label{res-est}
\sup_{\Re(z)\in\mathbb{R},\,\Im(z)>0}
\left|
\frac{d^k}{dz^k}
\mathbb{E}\!\left[
\langle \delta_0,(H_j^\omega-z)^{-1}\delta_0\rangle
\right]
\right|
\le
\tilde C_k,
\quad 0\le k\le m .
\end{equation}

\noindent Here the positive constant $\tilde C_k$ is independent of $\lambda_j\ge\lambda_0$.

\noindent Since $\lambda_0\le\lambda_j\le\tilde{\lambda}_0$, the spectrum satisfies $\sigma(H_j^\omega)\subset J \quad \text{for a.e.\ }\omega$,
where $J=[-2d+\lambda_0 a,\;2d+\tilde{\lambda}_0 b]$.
Consequently $g_j^{(k)}$ is continuous with support contained in $J$. Hence we may apply Lemma~\ref{drv-br} to obtain
\begin{align*}
\sup_{x\in\mathbb{R}}|g_j^{(k)}(x)|
&=
\sup_{x\in J}|g_j^{(k)}(x)|  \\
&\le
\sup_{\Re(z)\in J,\Im(z)>0}
\left|
\Im\!\left(\frac{d^k}{dz^k}F_j(z)\right)
\right|
+
|J|
\sup_{\Re(z)\in J,\Im(z)>0}
\left|
\Im\!\left(\frac{d^{k+1}}{dz^{k+1}}F_j(z)\right)
\right|.
\end{align*}

\noindent Define $C_k:=\max\{\tilde C_k,\tilde C_{k+1}\}$. Then \eqref{g-est} follows immediately from \eqref{brt} and \eqref{res-est}.
\end{proof}

\noindent We now estimate the decay of $\widehat{g_j}$, the Fourier transform of $g_j$.

\begin{corollary}
\label{de-fr}
Under the same assumptions as in Proposition~\ref{up-bd}, we have
\begin{equation}
\label{frdy}
\bigl|\widehat{g_j}(t)\bigr|
\le
\frac{C_{m-1}\bigl(1+4d+\tilde{\lambda}_0 b-\lambda_0 a\bigr)}{|t|^{m-1}},
\quad |t|>0,\quad j=1,2.
\end{equation}
\end{corollary}

\begin{proof}
Since $g_j^{(k)}\in C_c(\mathbb{R})$ for $0\le k\le m-1$, it follows that
\[
\widehat{g_j^{(k)}}(t)
=
(\mathrm{i}t)^k\,\widehat{g_j}(t),
\quad \mathrm{i}=\sqrt{-1}.
\]

\noindent In particular, for $k=m-1$ we obtain
\[
\widehat{g_j}(t)
=
\frac{1}{(\mathrm{i}t)^{m-1}}
\int_{\mathbb{R}}
e^{-\mathrm{i}tx}\,g_j^{(m-1)}(x)\,dx,
\quad j=1,2.
\]

\noindent Since $\operatorname{supp} g_j^{(m-1)} \subset [-2d+\lambda_0 a,\;2d+\tilde{\lambda}_0 b]=J$, the decay estimate \eqref{frdy} follows directly from \eqref{g-est}.
\end{proof}

\noindent We now prove the main result.

\begin{proof}[Proof of Theorem~\ref{main}]
For $j=1,2$, let $\nu_j(\cdot)
:=
\mathbb{E}\!\left[
\langle \delta_0, E_{H_j^\omega}(\cdot)\,\delta_0\rangle
\right]$ denote the density of states measure (DOSm) associated with $H_j^\omega$.
Define its Fourier transform by
\[
\widehat{\nu}_j(t)
:=
\int_{\mathbb{R}} e^{-\mathrm{i}tx}\,d\nu_j(x),
\quad \mathrm{i}=\sqrt{-1}.
\]

\noindent Using Duhamel’s formula (Theorem~\ref{DF} in the Appendix) and Fubini’s theorem, we obtain
\begin{align*}
|\widehat{\nu}_1(t)-\widehat{\nu}_2(t)|
&=
\Big|
\mathbb E\!\left[
\langle \delta_0,
(e^{-\mathrm{i}tH_1^\omega}-e^{-\mathrm{i}tH_2^\omega})
\delta_0\rangle
\right]
\Big| \\
&=
\left|
\int_0^t
\mathbb E\!\left[
\langle \delta_0,
e^{-\mathrm{i}(t-s)H_1^\omega}
(V_1^\omega-V_2^\omega)
e^{-\mathrm{i}sH_2^\omega}
\delta_0
\rangle
\right]ds
\right|.
\end{align*}

\noindent Expanding
\[
e^{-\mathrm{i}sH_2^\omega}\delta_0
=
\sum_{n\in\mathbb{Z}^d}
\langle\delta_n,e^{-\mathrm{i}sH_2^\omega}\delta_0\rangle
\,\delta_n,
\]
we obtain
\[
\langle \delta_0,
e^{-\mathrm{i}(t-s)H_1^\omega}
(V_1^\omega-V_2^\omega)
e^{-\mathrm{i}sH_2^\omega}
\delta_0
\rangle
=
\sum_{n\in\mathbb{Z}^d}
(\lambda_1-\lambda_2)\omega_n
\langle\delta_n,e^{-\mathrm{i}sH_2^\omega}\delta_0\rangle
\langle\delta_0,e^{-\mathrm{i}(t-s)H_1^\omega}\delta_n\rangle .
\]

\noindent Taking expectations, absolute values, and applying the Cauchy–Schwarz inequality,
\begin{align*}
&
\left|
\mathbb E\!\left[
\langle \delta_0,
e^{-\mathrm{i}(t-s)H_1^\omega}
(V_1^\omega-V_2^\omega)
e^{-\mathrm{i}sH_2^\omega}
\delta_0
\rangle
\right]
\right| \\
&\quad
\le
|\lambda_1-\lambda_2|
\sum_{n\in\mathbb{Z}^d}
\mathbb E\!\left[
|\omega_n|\,
|\langle\delta_n,e^{-\mathrm{i}sH_2^\omega}\delta_0\rangle|
|\langle\delta_0,e^{-\mathrm{i}(t-s)H_1^\omega}\delta_n\rangle|
\right] \\
&\quad
\le
M\,|\lambda_1-\lambda_2|\,
\mathbb E\!\left[
\sqrt{\sum_{n\in\mathbb{Z}^d}
|\langle\delta_n,e^{-\mathrm{i}sH_2^\omega}\delta_0\rangle|^2}
\sqrt{\sum_{m\in\mathbb{Z}^d}
|\langle\delta_0,e^{-\mathrm{i}(s-t)H_1^\omega}\delta_m\rangle|^2}
\right] \\
&\quad
=
M\,|\lambda_1-\lambda_2|\,
\mathbb E\!\left[
\|e^{-\mathrm{i}sH_2^\omega}\delta_0\|
\,
\|e^{\mathrm{i}(s-t)H_1^\omega}\delta_0\|
\right] \\
&\quad
=
M\,|\lambda_1-\lambda_2|.
\end{align*}

\noindent Here we used that the single site distribution is supported on $[a,b]$ (Hypothesis~\ref{hyp}), so that
$|\omega_n|\le M$ a.e. with $M:=\max\{|a|,|b|\}$ and the unitarity of the operators $e^{-\mathrm{i}tH_j^\omega}$ were then used to control the sums.

\noindent Consequently,
\begin{equation}
\label{ftme}
|\widehat{\nu}_1(t)-\widehat{\nu}_2(t)|
\le
M\,|\lambda_1-\lambda_2|\,|t|,
\quad \forall\, t\in\mathbb{R}.
\end{equation}

\noindent We have \( d\nu_j(x) = g_j(x)\,dx \); see Remark~\ref{cpt}.
Moreover, by Remarks~\ref{cpt} and~\ref{dr-cnt}, the integrated density of states $\mathcal{N}_j(x) := \nu_j((-\infty,x])
$ belongs to \( C^k(\mathbb{R}) \) for \(0\le k\le m\) and satisfies $\mathcal{N}_j^{(k+1)}(x)=g_j^{(k)}(x)\in C_c(\mathbb{R}), \quad 0\le k\le m-1$.

\noindent Thus, for every $x \in \mathbb{R}$ and $0 \le k < m - 2$, we obtain the estimate.
\begin{align*}
|\mathcal{N}_1^{(k+1)}(x)-\mathcal{N}_2^{(k+1)}(x)|
&=|g_1^{(k)}(x)-g_2^{(k)}(x)| \\
&=\left|\int_{\mathbb R} t^k e^{\mathrm{i}tx}
\big(\widehat{g}_1(t)-\widehat{g}_2(t)\big)\,dt\right|\\
&\le \int_{\mathbb R} |t|^k
|\widehat{g}_1(t)-\widehat{g}_2(t)|\,dt
=: I_1+I_2 .
\end{align*}

\noindent We split the integral at a parameter \(A>0\):
\[
I_1:=\int_{|t|\le A} |t|^k|\widehat{g}_1(t)-\widehat{g}_2(t)|\,dt,
\quad
I_2:=\int_{|t|>A} |t|^k|\widehat{g}_1(t)-\widehat{g}_2(t)|\,dt .
\]

\noindent Using \eqref{ftme}, we estimate \(I_1\) as
\[
I_1
=\int_{|t|\le A}|t|^k|\widehat{\nu}_1(t)-\widehat{\nu}_2(t)|dt
\le
M\int_{|t|\le A}|t|^{k+1}|\lambda_1-\lambda_2|dt
=
\frac{2MA^{k+2}}{k+2}|\lambda_1-\lambda_2|.
\]

\noindent For \(j=1,2\), the decay estimate \eqref{frdy} gives $|\widehat{g}_j(t)|
\le
\frac{C_{m-1}(1+4d+\tilde{\lambda}_0 b-\lambda_0 a)}{|t|^{m-1}},
~ t\neq0 .$

\noindent Hence
\begin{align*}
I_2
&\le
\int_{|t|\ge A}|t|^k|\widehat{g}_1(t)|dt
+
\int_{|t|\ge A}|t|^k|\widehat{g}_2(t)|dt \\
&\le
2\int_{|t|\ge A}
\frac{C_{m-1}(1+4d+\tilde{\lambda}_0 b-\lambda_0 a)}
{|t|^{m-k-1}}dt \\
&=
\frac{4C_{m-1}(1+4d+\tilde{\lambda}_0 b-\lambda_0 a)}
{(m-k-2)A^{m-k-2}} .
\end{align*}

\noindent Choose \(A=|\lambda_1-\lambda_2|^{-1/m}\) to balance \(I_1\) and \(I_2\).
Then
\[
|\mathcal{N}_1^{(k+1)}(x)-\mathcal{N}_2^{(k+1)}(x)|
\le
D_k|\lambda_1-\lambda_2|^{\frac{m-k-2}{m}},
\quad x\in\mathbb R ,
\]
where
\[
D_k=
\frac{2M}{k+2}
+
\frac{4C_{m-1}(1+4d+\tilde{\lambda}_0 b-\lambda_0 a)}{m-k-2},
\]
which is independent of \(\lambda_1\) and \(\lambda_2\).
Taking the supremum over \(x\) completes the proof.
\end{proof}

\appendix
\section{Appendix}
\noindent We recall the following form of Duhamel’s formula.

\begin{theorem}\label{DF}
Let $B$ be a self-adjoint operator on a Hilbert space $H$, let $V$ be a bounded self-adjoint operator, and set $A=B+V$. Then, for all $t\in\mathbb{R}$,
\begin{equation}
e^{-\mathrm{i}tA}-e^{-\mathrm{i}tB}
=
\int_0^t
e^{-\mathrm{i}(t-s)A}\, \mathrm{i}(B-A)\, e^{-\mathrm{i}sB}\,ds .
\end{equation}
\end{theorem}

\begin{proof}
The proof can be found in Lemma~5.2 of \cite{ACDS}.
\end{proof}

\noindent We also provide a bound for the Kantorovich--Rubinstein distance between the random variables $\lambda_1\omega_0$ and $\lambda_2\omega_0$.

\begin{definition}
The Kantorovich--Rubinstein (Wasserstein--1) distance between two probability measures
$\mu_1,\mu_2$ on $\mathbb{R}$ is defined by
\[
d_{KR}(\mu_1,\mu_2)
=
\sup\left\{
\left|
\int_{\mathbb{R}} f\,d\mu_1
-
\int_{\mathbb{R}} f\,d\mu_2
\right|
:
f:\mathbb{R}\to\mathbb{R}
\text{ is 1-Lipschitz}
\right\}.
\]
\end{definition}

\begin{proposition}
\label{kr}
Let $\mu_1$ and $\mu_2$ be the distributions of the random variables
$\lambda_1\omega_0$ and $\lambda_2\omega_0$, respectively. Then
\begin{equation}
\label{kr=df}
d_{KR}(\mu_1,\mu_2)
\le
|\lambda_1-\lambda_2|\,\mathbb{E}(|\omega_0|).
\end{equation}
\end{proposition}

\begin{proof}
We define the Lipschitz seminorm by
\[
\|f\|_{\mathrm{Lip}}
:=
\sup_{x\neq y}
\frac{|f(x)-f(y)|}{|x-y|}.
\]

\noindent Using the definition of the Kantorovich--Rubinstein distance, we obtain
\begin{align}
d_{KR}(\mu_1,\mu_2)
&=
\sup_{\|f\|_{\mathrm{Lip}}\le1}
\left|
\int_{\mathbb{R}} f\,d\mu_1
-
\int_{\mathbb{R}} f\,d\mu_2
\right| \nonumber\\
&=
\sup_{\|f\|_{\mathrm{Lip}}\le1}
\left|
\mathbb{E}\!\left[
f(\lambda_1\omega_0)-f(\lambda_2\omega_0)
\right]
\right| \nonumber\\
&\le
\sup_{\|f\|_{\mathrm{Lip}}\le1}
\mathbb{E}\!\left|
f(\lambda_1\omega_0)-f(\lambda_2\omega_0)
\right| \nonumber\\
&\le
\mathbb{E}\!\left|
\lambda_1\omega_0-\lambda_2\omega_0
\right| \nonumber\\
&=
|\lambda_1-\lambda_2|\,
\mathbb{E}(|\omega_0|).
\end{align}
\end{proof}

\end{document}